\pdfoutput=1
\documentclass[reqno]{amsart}
\usepackage{amsmath,amssymb,amsthm}
\usepackage{bm}
\usepackage{dsfont}
\usepackage{graphicx}
\usepackage{mathtools}
\usepackage{stmaryrd}

\newtheorem{theorem}{Theorem}[section]
\newtheorem{lemma}[theorem]{Lemma}

\newtheorem{corollary}[theorem]{Corollary}
\theoremstyle{definition}

\newtheorem*{remark}{Remark}

\newcommand{\ie}{i.\,e.,~}
\newcommand{\eg}{e.\,g.,~}

\renewcommand{\vec}{\mathbf}
\renewcommand{\epsilon}{\varepsilon}
\newcommand{\R}{\mathbb{R}}

\newcommand{\x}{\vec{x}}
\renewcommand{\u}{\vec{u}}
\newcommand{\enum}[1]{{[{#1}]}}

\newcommand{\dint}[1]{\,d{#1}}

\newcommand{\dx}{\dint{\x}}
\newcommand{\dy}{\dint{\y}}
\newcommand{\dz}{\dint{\z}}
\newcommand{\del}{\partial}
\newcommand{\grad}{\nabla}
\newcommand{\tr}{^\top}
\newcommand{\inv}{^{-1}}
\newcommand{\defas}{\coloneqq}
\newcommand{\asdef}{\eqqcolon}
\newcommand{\Prob}[1]{\mathbb{P}_{#1}}
\newcommand{\expct}[2]{\mathbb{E}_{#1}\left[#2\right]}
\newcommand{\prob}[2]{\mathbb{P}_{#1}\left(#2\right)}
\newcommand{\probthree}[3]{\mathbb{P}_{#1}^{#2}\left(#3\right)}
\newcommand{\var}[2]{\mathbb{V}\textnormal{ar}_{#1}\left({#2}\right)}

\newcommand{\std}[2]{\mathbb{S}\textnormal{td}_{#1}\left({#2}\right)}
\newcommand{\cv}[2]{\mathbb{C}\textnormal{V}_{#1}\left({#2}\right)}
\newcommand{\abs}[1]{\left\lvert{#1}\right\rvert}
\newcommand{\norm}[1]{\lVert{#1}\rVert}

\newcommand{\I}{\vec{I}}

\newcommand{\leb}{\bm{\lambda}}

\newcommand{\dhell}[2]{d_{\textnormal H}\left({#1},{#2}\right)}

\newcommand{\borel}[1]{\mathcal{B}({#1})}

\newcommand{\set}[2]{\left\lbrace{#1}\;\lvert\;{#2}\right\rbrace}

\newcommand{\diag}[1]{\textnormal{diag}{\left(#1\right)}}
\newcommand{\trace}[1]{\textnormal{tr}(#1)}

\newcommand{\diam}[1]{\textnormal{diam}({#1})}
\newcommand{\MSE}{\textnormal{MSE}}

\newcommand{\prior}{\rho_{\textnormal{prior}}}
\newcommand{\posterior}{\rho_{\textnormal{post}}}
\newcommand{\like}{\rho_{\textnormal{like}}}
\newcommand{\data}{{\vec{d}}}

\newcommand{\model}{\mathcal{G}}

\newcommand{\cov}{\Gamma}

\newcommand{\C}{\vec{C}}
\newcommand{\W}{\vec{W}}
\newcommand{\U}{\vec{U}}

\newcommand{\A}{\vec{A}}

\newcommand{\MeigVals}{\vec{\Lambda}}
\newcommand{\eigVal}{\lambda}
\newcommand{\eigVec}{\vec{w}}
\newcommand{\X}{{\vec{X}}}
\newcommand{\y}{{\vec{y}}}
\newcommand{\Y}{{\vec{Y}}}
\newcommand{\z}{{\vec{z}}}
\newcommand{\Z}{{\vec{Z}}}
\newcommand{\Wpert}{{\hat{\W}}}
\newcommand{\ypert}{{\hat{\y}}}
\newcommand{\Ypert}{{\hat{\Y}}}
\newcommand{\zpert}{{\hat{\z}}}
\newcommand{\Zpert}{{\hat{\Z}}}
\newcommand{\xyzW}[3]{{\llbracket {#1},{#2} \rrbracket_{#3}}}
\newcommand{\xyz}{\xyzW{\y}{\z}{}}

\newcommand{\posteriorgpert}{\rho_{\textnormal{post},\hat{g}}}

\newcommand{\posteriorgNpert}{\rho_{\textnormal{post},\hat{g}_N}}
\newcommand{\Xset}{\mathcal{X}}
\newcommand{\Yset}{\mathcal{Y}}
\newcommand{\Zset}{\mathcal{Z}}

\newcommand{\px}{\rho_{\X}}
\newcommand{\py}{\rho_{\Y}}
\newcommand{\pz}{\rho_{\Z}}
\newcommand{\pzy}{\rho_{\Z|\Y}}
\newcommand{\pyz}{\rho_{\Y,\Z}}
\newcommand{\pypert}{\rho_{\hat{\Y}}}

\newcommand{\pzypert}{\rho_{\hat{\Z}|\hat{\Y}}}

\usepackage[foot]{amsaddr}
\usepackage{enumitem}
\usepackage{geometry}
\usepackage[utf8]{inputenc}
\usepackage{subcaption}
\allowdisplaybreaks
\mathtoolsset{showonlyrefs=false}
\makeatletter
\renewcommand{\email}[2][]{%
	\ifx\emails\@empty\relax\else{\g@addto@macro\emails{,\space}}\fi%
	\@ifnotempty{#1}{\g@addto@macro\emails{\textrm{(#1)}\space}}%
	\g@addto@macro\emails{#2}%
}
\def\blfootnote{\gdef\@thefnmark{}\@footnotetext}
\makeatother
\title[A probabilistic framework for approximating functions in active subspaces]{A probabilistic framework for approximating functions\\in active subspaces}
\author[Mario Teixeira Parente]{Mario Teixeira Parente$^*$}
\address{$^*$Chair for Numerical Mathematics, Technical University Munich (TUM), Germany}
\email{parente@ma.tum.de}
\date{\today}
\keywords{dimension reduction, ridge approximation, conditional probability measure, Monte Carlo approximation}
\subjclass[2010]{65C60}
\numberwithin{equation}{section}
\begin{document}
\begin{abstract}
This paper develops a comprehensive probabilistic setup to compute approximating functions in active subspaces.
Constantine et\,al. proposed the \textit{active subspace method} in \cite{constantine2014active} to reduce the dimension of computational problems.
This method can be seen as an attempt to approximate a high-dimensional function of interest $f$ by a low-dimensional one.
A common approach for this is to integrate $f$ over the inactive, \ie non-dominant, directions with a suitable conditional density function.
In practice, this can be done using a finite Monte Carlo sum, making not only the resulting approximation random in the inactive variable, but also its expectation w.r.t. the active variable, \ie the integral of the low-dimensional function weighted with a probability measure on the active variable.
In this regard, we develop a fully probabilistic framework extending results from \cite{constantine2014active,constantine2016accelerating}.
The results are supported by a simple numerical example.
\end{abstract}
\maketitle
\blfootnote{\textbf{Funding:} This work was funded by the International Graduate School for Science and Engineering (IGSSE) of TUM.}
\setcounter{tocdepth}{1}
\tableofcontents
\section{Introduction}
The term \textit{active subspaces} refers to a recently emerging set of tools for dimension reduction \cite{constantine2014active}.
Reducing dimensions is one natural approach used in simplifying computational problems suffering from the \textit{curse of dimensionality}, a phenomenon that results in an exponential growth in computational costs with increasing dimensions.
What is regarded as a high dimension is dictated by the actual problem considered.
By "high", we mean a number of dimensions that lead to excessive computational times, the need of large memory, or even questions of feasibility of the computation.

There exist different approaches besides active subspaces in the reduction of computational effort, especially within the context of \textit{Uncertainty Quantification} (UQ) and \textit{Bayesian inversion} \cite{stuart2010inverse}.
For example, in \cite{buithanh2013computational,flath2011fast,spantini2015optimal} low-rank approximations for the \textit{prior-preconditioned Hessian} of the data misfit function were considered for approximating the posterior covariance in computationally intensive linear Bayesian inverse problems.
An extension for the nonlinear setting was proposed in \cite{martin2012stochastic}.
A drawback of these methods is that they still require work in the full, high-dimensional space.
A promising approach of dimension reduction was proposed in \cite{cui2014likelihood}, where dominant directions in the parameter space were sought to drive the update from the prior to the posterior distribution.
These directions span the so-called \textit{likelihood-informed subspace} (LIS) and are computed using the posterior expectation of the prior-preconditioned Hessian.
A study in \cite{zahm2018certified} develops a new methodology that constructs a \textit{controlled} approximation of the data misfit by a so-called \textit{profile function} composed with a low-rank projector such that an upper bound on the KL divergence between the posterior and the corresponding approximation falls under a given threshold.
These functions are similar to ridge functions \cite{pinkus2015ridge} and vary only on a low-dimensional subspace.
The upper bound is obtained via \textit{(subspace) logarithmic Sobolev inequalities} \cite{gross1975logarithmic} which is then easier accessible.
The fact that it controls the KL divergence makes it valuable since this quantity is often not available.
However, logarithmic Sobolev inequalities have rather strong assumptions excluding priors with compact support or heavy tails.
The paper also contains a comparison to likelihood-informed and active subspaces.
Most of the methods introduced work only for scalar-valued functions.
In \cite{zahm2018gradient}, a gradient-based dimension reduction was performed for functions in Hilbert spaces, \ie also for vector-valued functions.

Active subspaces also aim to find a ridge approximation for a function of interest $f$.
However, it exploits the structure of the function's gradient, more precisely, the (prior-)averaged outer product of the gradient with itself.
The technique was already successfully applied for a wide range of complex problems of engineering or economical relevance, \eg in hydrology \cite{jefferson2015}, for a lithium ion battery model \cite{constantine2017time}, or to an elastoplasticity model for methane hydrates \cite{teixeiraparente2018efficient}.

Independent of the concrete methodology, each approach on dimension reduction aims at unfolding the main and dominant information hidden in a low-dimensional structure.
Active subspaces concentrate on directions in a computational subdomain in which $f$ is more sensitive, on average, than in other (orthogonal) directions.
For that, the eigenvalues and corresponding eigenvectors of an uncentered covariance-type matrix defined by the average of the outer product of the gradient $\grad f$ with itself are studied.
The span of eigenvectors with corresponding large eigenvalues form the so-called \textit{active subspace}.
With the active subspace at hand, $f$ can be approximated by a low-dimensional ridge function depending on fewer variables.

A common approximation for $f$ uses a conditional expectation of $f$ over the complement of the active subspace, the \textit{inactive subspace}, conditioned on the \textit{active variable}, which is a linear combination of the variables from the original domain \cite{constantine2014active}.
In practice, the conditional expectation is often approximated using a finite Monte Carlo sum.
For this type of approximation, only a few samples are generally necessary since the function $f$ is, by construction, only mildly varying on the inactive subspace.

The idea of active subspaces was introduced in \cite{constantine2014active} and exploited for an accelerated Markov chain Monte Carlo algorithm \cite{brooks2011handbook} in \cite{constantine2016accelerating}.
Theoretical considerations therein ignore stochasticity in the inactive subspace.
This paper discerns this aspect and thus, performs a complete and rigorous analysis of approximating functions in active subspaces.
Eventually, we aim at providing a comprehensive probabilistic framework that generalizes the existing theoretical setting from \cite{constantine2014active,constantine2016accelerating}.
The findings are supported by a simple test example.

The manuscript is structured as follows. Section~\ref{sec:prob_form} formulates the mentioned problem generally without the notion of an active subspace.
Section~\ref{sec:as} explains and derives the concept of an active subspace in detail, and sets up a probabilistic setting for treating randomness in the inactive subspace.
Section~\ref{sec:approx} discusses the main results on approximating functions in active subspaces via a Monte Carlo approximation of a conditional expectation (see \eg Theorem~\ref{thm:var_mc} and Theorem~\ref{thm:var_mc_pert}).
In addition, this section presents a simple numerical example that verifies the theoretical results.
Section~\ref{sec:bayinv_as} restates and extends a result from \cite{constantine2016accelerating}.
Finally, Section~\ref{sec:summary} includes a summary and a collection of concluding comments.

\section{Problem formulation}
\label{sec:prob_form}
Suppose two random variables $\Y$ and $\Z$ follow a joint distribution with joint density $\pyz$.
Also, assume that the corresponding marginal and conditional densities are defined in the usual way \cite[Section~20~and~33]{billingsley1995probability}.
Let us define
\begin{equation}
	g(\y) \defas \int{f(\y,\z)\,\pzy(\z|\y)\dz}
\end{equation}
for a function $f$ which is integrable w.r.t. $\pzy(\cdot\,|\,\y)$ in its second argument for every $\y$.
We can approximate $g$ by a finite Monte Carlo sum
\begin{equation}
	g_N(\y) \defas \frac{1}{N}\sum_{j=1}^{N}f(\y,\Z^\y_j), \quad \Z^\y_j\sim\Prob{\Z|\Y}^\y,
\end{equation}
where $d\Prob{\Z|\Y}^\y/d\leb = \pzy(\cdot\,|\,\y)$ and $N>0$ denotes the number of Monte Carlo samples used for each $\y$.
This means that $g_N$ is a random variable for every $\y$.
Now, assume that a high-dimensional variable $\x$ is divided into $\y$ and $\z$, i.\,e. $\x \mapsto (\y,\,\z)$.
For convenience, let us construct a function
\begin{equation}
	f_{g_N}(\x) \defas g_N(\y)
\end{equation}
defined on the corresponding high-dimensional domain.

The first main point of this manuscript is to give a rigorous description, in the context of active subspaces, of why the expression
\begin{equation}
	\label{eq:expct_fgn}
	\expct{}{f_{g_N}(\X)} = \int{f_{g_N}(\x)\,\px(\x)\dx},
\end{equation}
where $\px(\x)=\pyz(\y,\z)$, is in general \textit{random}.
The second point deals with the consequences (of treating this expression as non-deterministic) that lie in expanding the results from \cite{constantine2014active,constantine2016accelerating} in the given probabilistic framework.

\section{Active subspaces}
\label{sec:as}
Active subspaces, introduced in \cite{constantine2014computing,constantine2014active,constantine2016accelerating}, try to find a ridge approximation \cite{constantine2016many,pinkus2015ridge} to a measurable function $f:\Xset\to\R$, $\Xset\subseteq\R^n$ open, \ie $f(\x) \approx g(\mathbf{A}\tr\x)$ for all $\x\in\Xset$ by a measurable function $g:\Yset\to\R$, $\Yset\subseteq\R^k$ and a matrix $\mathbf{A}\in\R^{n\times k}$.
Obviously, it is hoped that $k \ll n$ to sufficiently reduce the dimension.
$\mathbf{A}$ is computed to hold the directions in which $f$ is more sensitive, on average, than in other directions.
This means that $f$ is nearly insensitive, on average, for directions in the null space of $\A\tr$ since $f(\x+\vec{w})\approx g(\A\tr(\x+\vec{w})) = g(\A\tr\x)\approx f(\x)$ for each $\vec{w}\in\mathcal{N}(\A\tr)\defas\set{\vec{v}\in\R^n}{\A\tr\vec{v}=\vec{0}}$.
The notion "on average" is crucial in the succeeding statements and means that sensitivities are weighted with a probability density function $\px$ defined on $\R^n$.

Now, let $\Xset$ denote the set of all $\x$'s with a positive density value, i.\,e.
\begin{equation}
	\Xset \defas \left\lbrace \x\in\R^n\;\lvert\;\px(\x)>0 \right\rbrace.
\end{equation}
We assume that $\Xset$ is open and hence $\Xset\in\borel{\R^n}$.
Thus, $\px$ is assumed to be zero on the boundary $\del\Xset$.
Also, suppose that $\Xset$ is a \textit{continuity set}, \ie $\leb^n(\del\Xset)=0$.
We will often make use of the fact that it is enough to integrate over $\Xset$ instead of $\R^n$ when weighting with $\px$.
In order to find the matrix $\mathbf{A}$ for the ridge approximation, we assume that $f$ has partial derivatives that are square integrable w.r.t. $\px$.
Additionally, we assume that $\px$ is \textit{bounded} and \textit{continuous} on $\Xset$.

To study sensitivities, we regard an orthogonal eigendecomposition of the averaged outer product of the gradient $\grad f : \Xset\to\R^n$ with itself, i.\,e.
\begin{equation}
	\label{eq:C}
	\C \defas \int_{\Xset}{\grad f(\x) \grad f(\x)\tr \px(\x) \dx} = \W\bm{\Lambda}\W\tr,
\end{equation}
where $\bm{\Lambda} = \text{diag}(\eigVal_1,\ldots,\eigVal_n)$ denotes the eigenvalue matrix with descending eigenvalues and $\W=\left[\eigVec_1 \cdots \eigVec_n\right]$ consists of all corresponding normed eigenvectors.
The fact that $\C$ is real symmetric implies that the eigenvectors $\eigVec_i$ can be chosen to give an \textit{orthonormal basis} (ONB) of $\R^n$.
Since $\C$ is additionally positive semi-definite, it holds that $\eigVal_i\geq0$, $i\in\enum{n}\defas\{1,\dots,n\}$.
Note that the eigenvalues
\begin{equation}
	\eigVal_i = \eigVec_i\tr\C\eigVec_i = \int_{\Xset}{\left(\eigVec_i\tr\grad f(\x)\right)^2\px(\x)\dx}, \quad i\in\enum{n},
\end{equation}
reflect the averaged sensitivities of $f$ in the direction of the corresponding eigenvectors.
That means that $f$ changes little, on average, in the directions of eigenvectors with small eigenvalues.

If it is possible to find a sufficiently large spectral gap, we can accordingly split $\W=\left[\W_1\;\;\W_2\right]$.
That is, $\W_1\in\R^{n\times k}$, $k\in\enum{n-1}$, holds the directions for which $f$ is more sensitive, on average, than for directions in $\W_2\in\R^{n\times(n-k)}$.
Dimension $k$ denotes the number of eigenvalues \textit{before} the spectral gap.
The size of the gap is crucial for the approximation quality of the active subspace \cite{constantine2014computing}.
After splitting $\W$, we can get a new parametrization of $\x$ such that
\begin{equation}
	\x = \W\W\tr\x = \W_1\W_1\tr\x +\W_2\W_2\tr\x = \W_1\y+\W_2\z,
\end{equation}
with $\y\defas\W_1\tr\x$, $\z\defas\W_2\tr\x$.
The variable $\y$ is called the \textit{active variable} and the column space of $\W_1$, $\mathcal{R}(\W_1) \defas \left\lbrace \W_1\y \;|\; \y\in\R^k \right\rbrace$, the \textit{active subspace}.

\textbf{Notation}
Throughout the remainder, we use some notation to avoid uninformative text.
From the previous lines, we can define $\xyz\defas\xyzW{\y}{\z}{\W}\defas\W_1\y+\W_2\z$ to shorten texts in the equations that follow.
Also, for a compatible pair of a matrix $A$ and a set $\mathcal{V}$, we define $A\mathcal{V}\defas\set{Av}{v\in\mathcal{V}}$.
Additionally, for a set $\mathcal{V}\subseteq\R^n$, we will set $\Yset_\mathcal{V}\defas\W_1\tr\mathcal{V}$, \ie $\Yset_\mathcal{V}$ is the set of $\y$-coordinates of points in $\mathcal{V}$.

\subsection*{Probabilistic setting}
Let $(\Omega,\mathcal{A},\Prob{})$ be an abstract probability space.
The random variable $\X:\Omega\to\R^n$ stands for $\x\in\R^n$ viewed as a random element whose push-forward measure $\Prob{\X} \defas \prob{}{\X\in\cdot}$ has Lebesgue density $d\Prob{\X}/d\leb = \px$.
The random variables $\Y \defas \W_1\tr\X$ and $\Z \defas \W_2\tr\X$ representing random elements in the active and inactive subspaces also induce corresponding push-forward measures $\Prob{\Y} \defas \prob{}{\Y\in\cdot}$ and $\Prob{\Z} \defas \prob{}{\Z\in\cdot}$.
It is possible to define a joint probability density function for the active and inactive variables with $\px$, \ie
\begin{equation}
\px(\x) = \px(\xyz) \asdef \pyz(\y,\z).
\end{equation}
Note that $\pyz$ inherits boundedness from $\px$.
The marginal densities are also defined in the usual way \cite[Section~20~and~33]{billingsley1995probability}, \ie
\begin{equation}
	\py(\y) \defas \int_{\R^{n-k}}{\pyz(\y,\z)\dz}
\end{equation}
and
\begin{equation}
	\pz(\z) \defas \int_{\R^k}{\pyz(\y,\z)\dy}.
\end{equation}
Note that
\begin{equation}
	\frac{d\Prob{\Y}}{d\leb} = \py \quad\text{and}\quad \frac{d\Prob{\Z}}{d\leb} = \pz.
\end{equation}
For convenience, we will define domains for the active and inactive variables, \ie
\begin{equation}
	\Yset \defas \W_1\tr\Xset \subseteq\R^k \quad\text{and}\quad \Zset \defas \W_2\tr\Xset \subseteq\R^{n-k}.
\end{equation}
Note that $\Yset$ will be the domain of the low-dimensional function $g$ approximating $f$.

The lemma that follows shows that $\Yset$ and $\Zset$ can be characterized as sets of vectors in the active and inactive subspaces, respectively, with positive marginal density values.
Therefore, let
\begin{equation}
	\Yset^*\defas\left\lbrace \y\in\R^k\;\lvert\; \py(\y)>0 \right\rbrace \quad  \text{and} \quad \Zset^*\defas\left\lbrace \z\in\R^{n-k}\;\lvert\; \pz(\z)>0 \right\rbrace.
\end{equation}
We need the result for a proper definition of conditional densities.
\begin{lemma}
\label{lem:Yset_py}
It holds that
\begin{equation}
	\Yset = \Yset^* \quad \text{and} \quad \Zset = \Zset^*.
\end{equation}
\end{lemma}
\begin{proof}
We will only show the proof for the equality for $\Yset$ since the same arguments follow for $\Zset$.
Let us take an arbitrary $\y\in\Yset$ and choose $\x\in\Xset$ such that $\y=\W_1\tr\x$.
Set $\z'\defas\W_2\tr\x$ and $\rho\defas\pyz(\y,\z')>0$.
Due to the openness of $\Xset$ and the continuity of $\px$ on $\Xset$, we can find a neighborhood $\Zset_\epsilon$ of $\z'$ such that
\begin{equation}
\px(\xyzW{\y}{\z_\epsilon}{}) \geq \frac{\rho}{2}
\end{equation}
for each $\z_\epsilon\in\Zset_\epsilon$.
It follows that
\begin{equation}
	\py(\y) = \int_{\R^{n-k}}{\pyz(\y,\z)\dz} \geq \int_{\Zset_\epsilon}{\pyz(\y,\z_\epsilon)\dz_\epsilon} \geq \frac{\rho}{2}\leb^{n-k}(\Zset_\epsilon) > 0,
\end{equation}
and thus $\y\in\Yset^*$.
Reversely, let us choose $\y^*\in\Yset^*$.
Since $\py(\y^*) > 0$, it exists a $\z^*\in\R^{n-k}$ such that $\x^* \defas \xyzW{\y^*}{\z^*}{} \in\Xset$.
Since $\y^* = \W_1\tr\x^*$, it follows that $\y^*\in\Yset$.
\end{proof}

\begin{corollary}
It holds that $\Prob{\Y}(\Yset)=1$ and $\Prob{\Z}(\Zset)=1$, \ie $\Y\in\Yset$ a.s. and $\Z\in\Zset$ a.s.
\end{corollary}

\begin{lemma}
The sets $\Yset\subseteq\R^k$ and $\Zset\subseteq\R^{n-k}$ are open in respective topological spaces.
\end{lemma}
\begin{proof}
We will only show the proof for the openness for $\Yset$ since the same arguments follow for $\Zset$.
Let $\y_0\in\Yset$.
By definition, there exists an  $\x_0\in\Xset$ with $\y_0=\W_1\tr\x_0$.
Since $\Xset$ is open, there exists a ball $B(\x_0)\subseteq\Xset$.
Since $\y_0\in\Yset_{B(\x_0)}$, it suffices to show that $\Yset_{B(\x_0)}\subseteq\Yset$.

Now, let us take $\y\in\Yset_{B(\x_0)}$.
We can compute
\begin{align}
\py(\y) &= \int_{\Zset}{\pyz(\y,\z)\dz} \\
&\geq \int_{\set{\z\in\Zset}{\xyz\in B(\x_0)}}{\pyz(\y,\z)\dz} > 0.
\end{align}
Since $\W$ is orthogonal, it causes only a rotation.
The set $\set{\z\in\Zset}{\xyz\in B(\x_0)}$ has a positive measure under $\leb^{n-k}$ which justifies the last equation above.
The result follows by Lemma~\ref{lem:Yset_py}.
\end{proof}
In particular, the previous lemma implies that $\Yset\in\borel{\R^k}$ and $\Zset\in\borel{\R^{n-k}}$.
Another auxiliary result shows that it is enough for the marginal densities to integrate over $\Yset$ and $\Zset$, respectively.
\begin{lemma}
\label{lem:marg_on_YZ}
It holds that
\begin{equation}
	\py(\y) = \int_{\Zset}{\pyz(\y,\z)\dz}, \quad \y\in\R^k, \quad\text{and}\quad\pz(\z) = \int_{\Yset}{\pyz(\y,\z)\dy}, \quad \z\in\R^{n-k}.
\end{equation}
\end{lemma}
\begin{proof}
We will only show the proof for the equality for $\py$ since the same arguments follow for $\pz$.
Let $\y\in\R^k$.
We can write
\begin{equation}
\label{eq:pyz_Rnk_Z}
	\py(\y)=\int_{\Zset}{\pyz(\y,\z)\dz} + \int_{\R^{n-k}\setminus\Zset}{\pyz(\y,\z)\dz}.
\end{equation}
For $\z\in\R^{n-k}\setminus\Zset$, it holds that $\xyz\not\in\Xset$; otherwise, $\z\in\Zset$, which is contradictory.
It follows that in \eqref{eq:pyz_Rnk_Z},
\begin{equation}
	\int_{\R^{n-k}\setminus\Zset}{\pyz(\y,\z)\dz} = \int_{\R^{n-k}\setminus\Zset}{\px(\xyz)\dz} = 0
\end{equation}
implying the desired result.
\end{proof}

As a consequence of Lemma~\ref{lem:Yset_py}, we are able to define a proper conditional density on $\R^{n-k}$ given $\y\in\Yset$, \ie
\begin{equation}
	\pzy(\z|\y) \defas \frac{\pyz(\y,\z)}{\py(\y)}, \quad \z\in\R^{n-k}.
\end{equation}
Note that $\pzy(\z|\y) = 0$ for $\z\not\in\Zset$ and arbitrary $\y\in\Yset$ as shown in the previous proof.
Thus, it is possible to define a regular conditional probability distribution of $\Z$ given $\Y=\y$ for $\y\in\Yset$,
\begin{equation}
	\label{eq:prob_Zy}
	\Prob{\Z|\Y}^\y \defas \prob{}{\Z\in\cdot \;|\; \Y=\y}.
\end{equation}
For details of the construction, see \eg \cite{durrett2010probability}.
This can be connected to the respective conditional density by
\begin{equation}
	\frac{d\Prob{\Z|\Y}^\y}{d\leb} = \pzy(\cdot\,|\,\y).
\end{equation}
for $\y\in\Yset$.
For $\y\not\in\Yset$, let us define $\pzy(\z|\y)=0$ for all $\z\in\R^{n-k}$.
The random variable $\Z^\y \sim \Prob{\Z|\Y}^\y$, $\y\in\Yset$, is drawn from the conditional probability distribution defined in \eqref{eq:prob_Zy}.
However, it will be necessary to also regard $\y$ as random which we denote with the random variable $\Y$.
That is, the measure where $\Z^\Y$ is drawn from is also random.
An abstract framework to deal with in this context is known as \textit{random measure} (see \eg \cite{kallenberg2017random}).

In order to apply Fubini's theorem, which requires product measurability of the function to be integrated, in Theorem~\ref{thm:var_mc} and \ref{thm:var_mc_pert}, we need to prove a measurability result for the map
\begin{align}
	\Z^\y : (\Yset\times\Omega,\, \borel{\Yset}\otimes\mathcal{A}) &\to (\R^{n-k},\,\borel{\R^{n-k}}), \label{eq:Zy} \\
	(\y,\omega) &\mapsto \Z^\y(\omega).
\end{align}
The result will be used in Lemma~\ref{lem:gN_meas} to obtain a product measurable function.
Note that we regard $\Yset\subseteq\R^k$ as a topological space equipped with the usual subspace topology denoted by $\borel{\Yset}$.

\begin{lemma}
\label{lem:Zy_pr_meas}
The map $(\y,\omega) \mapsto \Z^\y(\omega)$ is $\borel{\Yset}\otimes\mathcal{A}$-measurable.
\end{lemma}
\begin{proof}
Let $\y_0\in\Yset$.

For the moment, assume that $\Z$ is real-valued and change its notation to $Z$.
Let $F^{\y_0} : \R\to[0,1]$ denote the cumulative distribution function of $Z^{\y_0}$.
Note that the map $\y\mapsto F^\y(t)$ is $\borel{\Yset}$-measurable for each $t\in\R$.
Now, let $t\in\R$.
Indeed, for $\y\in\Yset$, it holds that
\begin{equation}
	F^\y(t) = \frac{\int_{-\infty}^{t}{\pyz(\y,z)\dint{z}}}{\int_{-\infty}^{\infty}{\pyz(\y,z)\dint{z}}}.
\end{equation}
The measurability follows from the product measurability of $\pyz$.
By the probability integral transform, we can write
\begin{equation}
	Z^{\y_0} = G^{\y_0}(U),
\end{equation}
where $U\sim\mathcal{U}([0,1])$ and $G^{\y_0} : [0,1]\to\R$ is the (generalized) inverse distribution function of $F^{\y_0}$.
Hence, it suffices to show the product measurability of $(\y,u) \mapsto G^\y(u)$.
It holds that
\begin{equation}
	\set{(\y,u)\in\Yset\times[0,1]}{G^\y(u)\leq t} = \set{(\y,u)\in\Yset\times[0,1]}{u \leq F^\y(t)} \in \borel{\Yset\times[0,1]}.
\end{equation}
The last step follows from the measurability of $\y\mapsto F^\y(t)$ and the fact that $h_t(\y,u)\defas F^\y(t)-u$ is $\borel{\Yset\times[0,1]}$-measurable.

Now, let us assume that $\Z$ is again $\R^{n-k}$-valued.
Also, let $F^{\y_0}_i$ denote the cumulative distribution function of $Z^{\y_0}_i$, $i\in\enum{n-k}$.
Similar to the one-dimensional case, the map $\y\mapsto F_i^\y(t)$ is $\borel{\Yset}$-measurable for each $t\in\R$, $i\in\enum{n-k}$.
Again, we can write
\begin{equation}
	\Z^{\y_0} = \mathbf{G}^{\y_0}(\U),
\end{equation}
where $\U\sim C_{\Z^{\y_0}}$ and
\begin{equation}
	\mathbf{G}^{\y_0} : [0,1]^{n-k}\to\R^{n-k},\; \u\mapsto\left((F^{\y_0}_1)\inv(u_1),\dots,(F^{\y_0}_{n-k})\inv(u_{n-k})\right).
\end{equation}
The expression $C_{\Z^{\y_0}}$ is called a \textit{copula distribution} of $\Z^{\y_0}$ \cite{nelsen2007introduction} and $(F^{\y_0}_i)\inv$ is the (generalized) inverse distribution of $F^{\y_0}_i$, $i\in\enum{n-k}$.
Hence, it suffices to show the product measurability of $(\y,\u) \mapsto \mathbf{G}^\y(\u)$ by noting that the map $\pi_i(\u)\defas u_i$ is measurable and by applying the steps from the one-dimensional case component-wise.
It follows that $(\y,\omega) \mapsto \Z^\y(\omega)$ is $\borel{\Yset}\otimes\mathcal{A}$-measurable.
\end{proof}

\textbf{Notation}
It is important to clarify some notations that is used throughout the remainder.
We will use three different expectations for the integration of random variables $\X$, $\Y$ and $\Z^\y$.
The respective expectations will be denoted by $\mathbb{E_\X}$, $\mathbb{E}_\Y$ and $\mathbb{E}_\Z$.

Also, we will oftentimes use a change of variables from $\x$ to $(\y,\z)$ during integration.
For that, a useful statement used frequently is proved in the subsequent lemma.

\begin{lemma}
\label{lem:x_yz}
For any real-valued function $h\in L^1(\Omega,\mathcal{A},\Prob{})$, it holds that
\begin{equation}
	\expct{\X}{h(\X)} = \expct{\Y}{\expct{\Z}{h(\xyzW{\Y}{\Z^\Y}{})}}.
\end{equation}
\end{lemma}
\begin{proof}
Define $\Phi(\y,\z) \defas \xyz = \x$.
Note that $\grad_{\y,\z}(\Phi(\y,\z)) = \left[\W_1\;\;\W_2\right] = \W$.
Integration by substitution for multiple variables gives
\begin{align}
	\expct{\X}{h(\X)} &= \int_{\Xset}{h(\x)\,\px(\x)\dx} \\
	&= \int_{\Yset}{\int_{\Zset}{h(\Phi(\y,\z))\,\px(\Phi(\y,\z))\abs{\det(\W)} \dz} \dy} \\
	&= \int_{\Yset}{\int_{\Zset}{h(\xyz)\,\pyz(\y,\z) \dz} \dy} \label{eq:detW} \\
	&= \int_{\Yset}{\left(\int_{\Zset}{h(\xyz)\,\pzy(\z|\y) \dz}\right) \,\py(\y) \dy} \\
	&= \int_{\Yset}{\left(\int_{\Zset}{h(\xyz) \dint{\Prob{\Z|\Y}^\y}(\z)}\right) \dint{\Prob{\Y}(\y)}} \\
	&= \int_{\Omega}{\left(\int_{\Omega}{h(\xyzW{\Y(\omega_\Y)}{\Z^{\Y(\omega_\Y)}(\omega)}{}) \dint{\Prob{}(\omega)}}\right) \dint{\Prob{}(\omega_\Y)}} \\
	&= \expct{\Y}{\expct{\Z}{h(\xyzW{\Y}{\Z^\Y}{})}}.
\end{align}
In \eqref{eq:detW}, we use that $\det(\W)=\pm 1$ for the orthogonal matrix $\W$.
\end{proof}
\section{Approximating functions in the active subspace}
\label{sec:approx}
Once the active subspace is computed, the function $f$ can be approximated by a function $g$ on a lower-dimensional domain.
One way to define a suitable approximation is by a conditional expectation, i.\,e.
\begin{align}
	g(\y) \defas{}& \expct{\X}{f(\X)\,|\,\Y=\y} \\
	\defas{}& \int_{\R^{n-k}}{f(\xyz) \dint{\Prob{\Z|\Y}^\y}(\z)} \\
	={}& \int_{\R^{n-k}}{f(\xyz)\,\pzy(\z|\y) \dz} \\
	={}& \int_{\Zset}{f(\xyz)\,\pzy(\z|\y) \dz} \label{eq:gy}
\end{align}
for $\y\in\Yset$.
Note that the last line is justified by the fact that $\pzy(\z|\y) = 0$ for $\z\not\in\Zset$ and arbitrary $\y\in\Yset$.
The conditional expectation is known to be the best $L^2$ approximation of $f$ in the active subspace \cite{billingsley1995probability}.
To obtain a function on the same domain as $f$, \ie $\Xset$, let us define
\begin{equation}
	\label{eq:f_g}
	f_g(\x) \defas g(\W_1\tr\x).
\end{equation}

In practice, the weighted integral in \eqref{eq:gy} can be approximated using a finite Monte Carlo sum with
\begin{equation}
	\label{eq:gNy}
	g_N(\y,\cdot) \defas \frac{1}{N} \sum_{j=1}^{N}{f(\xyzW{\y}{\Z^\y_j(\cdot)}{})}, \quad \Z^\y_j\sim\Prob{\Z|\Y}^\y,\;N>0.
\end{equation}
for $\y\in\Yset$.
Note that $g_N(\y,\cdot)$ is again random for every $\y\in\Yset$.
Similar to \eqref{eq:f_g}, we can define a suitable function on $\Xset\times\Omega$ such that
\begin{equation}
	f_{g_N}(\x,\cdot) \defas g_N(\W_1\tr\x,\cdot).
\end{equation}
It is important to recognize the following relationship between the expectations of $f_{g_N}(\X)$ and $g_N(\Y)$.
It holds that
\begin{equation}
	\expct{\X}{f_{g_N}(\X,\cdot)} = \expct{\Y}{\expct{\Z}{f_{g_N}(\xyzW{\Y}{\Z^\Y}{},\cdot)}} = \expct{\Y}{\expct{\Z}{g_N(\Y,\cdot)}} = \expct{\Y}{g_N(\Y,\cdot)}.
\end{equation}
This equation is a crucial point in this manuscript as it makes it clear that both expressions $\expct{}{f_{g_N}(\X,\cdot)}$ and $\expct{\Y}{g_N(\Y,\cdot)}$ are random.
More explicitly, this can be seen in the following equations.
Let $\omega\in\Omega$ be fixed, then
\begin{align}
	\expct{\X}{f_{g_N}(\X,\omega)} &= \int_{\Xset}{f_{g_N}(\x,\omega)\,\px(\x)\dx} \\
	&= \int_{\Yset}{\left(\int_{\Zset}{f_{g_N}(\xyz,\omega)\,\pzy(\z|\y)\dz}\right)\py(\y)\dy} \\
	&= \int_{\Yset}{\left(\int_{\Zset}{g_N(\y,\omega)\,\pzy(\z|\y)\dz}\right)\py(\y)\dy} \label{eq:loose_z} \\
	&= \int_{\Yset}{g_N(\y,\omega)\,\py(\y)\dy} \\
	&= \expct{\Y}{g_N(\Y,\omega)}.
\end{align}
Note that in \eqref{eq:loose_z}, the variable $\z$, "belonging" to $\x$, disappears such that the integral w.r.t. $\z$ becomes $\int_{\Zset}{\pzy(\z|\y)\dz}=1$. The random variables $\Z_j^\y$ within $g_N$ are \textit{not} integrated over $\z$, \ie the variables $\Z_j^\y$ are not bound in terms of formal languages.
This leads to the fact that $\expct{\X}{f_{g_N}(\X,\cdot)}$ is again random.

We can now regard the expressions $\expct{}{\expct{\Y}{g_N}}$ or $\expct{\Y}{\expct{}{g_N}}$.
We will show that both are equal and thus, we will regard the first, \ie $\expct{}{\expct{\Y}{g_N}}$.
In the proof of Theorem~\ref{thm:var_mc}, we thus computed the \textit{expectation of the mean squared error} between $f_g$ and $f_{g_N}$ and had to change the order of integration w.r.t. $\y$ and $\omega$, \ie apply Fubini's theorem.
To do this properly, we need to show the measurability of $g_N$ in the product space $\Yset\times\Omega$.

\begin{lemma}
\label{lem:gN_meas}
The function $g_N: \Yset\times\Omega\to\R$ is $\mathcal{B(\Yset)}\otimes\mathcal{A}$-measurable.
\end{lemma}
\begin{proof}
Lemma~\ref{lem:Zy_pr_meas} proves the product measurability of $\Z^\y$, defined in \eqref{eq:Zy}.
This implies the result.
\end{proof}

One important result, already proved in \cite[Theorem~3.1]{constantine2014active}, gives an upper bound on the mean squared error of $f_g$ by the eigenvalues corresponding to the inactive subspace.
The proof is motivated by a Poincaré inequality for the integral over the inactive subspace.
We have to be careful with the corresponding constant, the \textit{Poincaré constant}, that is depending on the active variable $\y$.
Hence, for the sake of simplicity, we prove the following theorem for two special cases:
\begin{enumerate}
	\item \label{it:unif_conv} $\Prob{\X}=\mathcal{U}(\Xset)$ for $\Xset$ being bounded and convex.
	\item \label{it:normal} $\Prob{\X}=\mathcal{N}(\vec{0},\I)$, \ie $\Xset=\R^n$.
\end{enumerate}

\begin{remark}
	The theorem is actually valid under much more general conditions.
	For example, the authors of \cite{zahm2018certified} use the theory of \textit{logarithmic Sobolev inequalities} allowing weaker assumptions.
	Also, it is known that the Poincaré inequality is valid for measures satisfying a \textit{Muckenhoupt condition} \cite{turesson2007nonlinear}.
\end{remark}

\begin{theorem}
\label{thm:expct_fg_f}
For the cases \eqref{it:unif_conv} and \eqref{it:normal} from above, it holds that
\begin{equation}
	\expct{\X}{\left(f(\X)-f_g(\X)\right)^2} \leq C_{\ref{thm:expct_fg_f}} \left(\lambda_{k+1}+\cdots+\lambda_n\right),
\end{equation}
for some constant $C_{\ref{thm:expct_fg_f}} > 0$.
\end{theorem}
\begin{proof}
Note that
\begin{equation}
	\label{eq:poinc_mean_zero}
	\expct{\Z}{f\left(\xyzW{\y}{\Z^{\y}}{}\right)-g(\y)} = 0
\end{equation}
for every $\y\in\Yset$.
It follows that
\begin{align}
	\expct{\X}{\left(f(\X)-f_g(\X)\right)^2} &= \expct{\Y}{\expct{\Z}{\left(f\left(\xyzW{\Y}{\Z^\Y}{}\right)-g(\Y)\right)^2}} \label{eq:x_yz} \\
	&= \int_{\Yset}{\left(\int_{\Zset}{\left(f\left(\xyzW{\y}{\z}{}\right)-g(\y)\right)^2 \pzy(\z|\y)\dz}\right) \py(\y)\dy} \\
	&\leq \int_{\Yset}{C_\y \left(\int_{\Zset}{\grad_\z f\left(\xyzW{\y}{\z}{}\right)\tr \grad_\z f\left(\xyzW{\y}{\z}{}\right) \pzy(\z|\y)\dz}\right) \py(\y)\dy} \\
	&\leq C_{\ref{thm:expct_fg_f}} \int_{\Yset}{\left(\int_{\Zset}{\grad_\z f\left(\xyzW{\y}{\z}{}\right)\tr \grad_\z f\left(\xyzW{\y}{\z}{}\right) \pzy(\z|\y)\dz}\right) \py(\y)\dy} \\
	&\leq C_{\ref{thm:expct_fg_f}} \, \expct{\Y}{\expct{\Z}{\grad_\z f\left(\xyzW{\Y}{\Z^\Y}{}\right)\tr \grad_\z f\left(\xyzW{\Y}{\Z^\Y}{}\right)}} \label{eq:poincare} \\
	&= C_{\ref{thm:expct_fg_f}} \, \expct{\X}{\grad_\z f(\X)\tr \grad_\z f(\X)} \\
	&\leq C_{\ref{thm:expct_fg_f}} \left(\lambda_{k+1}+\cdots+\lambda_n\right),
\end{align}
where $C_{\ref{thm:expct_fg_f}} \defas \sup_{\y\in\Yset}{C_\y}$.
In \eqref{eq:x_yz}, we use Lemma~\ref{lem:x_yz} while \eqref{eq:poincare} uses a Poincaré inequality w.r.t. the inactive subspace and $\pzy$ which is applicable due to \eqref{eq:poinc_mean_zero}.
The last line follows from \cite[Lemma~2.2]{constantine2014active}.

Let $\y\in\Yset$.
For case \eqref{it:unif_conv}, \cite{bebendorf2003poincare} proves that $C_\y=\diam{\Zset_\y}/\pi$, where
\begin{equation}
	\Zset_\y\defas\set{\z\in\R^{n-k}}{\pyz(\y,\z)>0}.
\end{equation}
Observe that $C_\y \leq C_{\ref{thm:expct_fg_f}} \leq\diam{\Zset}/\pi = $ since $\Zset_\y\subseteq\Zset$.
It holds that $\diam{\Zset}<\infty$ because $\Xset$ is assumed to be bounded.

The constant for case \eqref{it:normal} can proven to be $C_{\ref{thm:expct_fg_f}}=1$ \cite{chen1982inequality}.
This is possible since $\pzy$ is again the density of the standard normal distribution which follows by its rotational symmetry.
\end{proof}

This means, that if all eigenvalues corresponding to the inactive subspace are small or even zero, then the mean squared error of the conditional expectation is also small or zero.
In contrast, if the inactive subspace is spanned by too many eigenvectors with rather large corresponding eigenvalues, then the approximation might be poor.

Lemma~\ref{lem:gN_meas} not only proves that $\expct{\Y}{g_N(\Y,\cdot)}$ is indeed a random variable, \ie a \textit{measurable} map from $\Omega$ to $\R$, but also suggests that we are ready to prove a crucial result with it.
The main difference to \cite[Theorem~3.2]{constantine2014active} is that the result is an upper bound on the \textit{expectation} of the mean squared error of $f_{g_N}$ to $f_g$.
\begin{theorem}
\label{thm:var_mc}
Under the assumptions of Theorem~\ref{thm:expct_fg_f}, it holds that
\begin{equation}
	\expct{}{\expct{\X}{\left(f_g(\X)-f_{g_N}(\X,\cdot)\right)^2}} = \expct{}{\expct{\Y}{\left(g(\Y)-g_N(\Y,\cdot)\right)^2}} \leq \frac{C_{\ref{thm:expct_fg_f}}}{N}(\lambda_{k+1}+\cdots+\lambda_n).
\end{equation}
\end{theorem}
\begin{proof}
For fixed $\omega\in\Omega$,
\begin{align}
\expct{\X}{\left(f_g(\X)-f_{g_N}(\X,\omega)\right)^2} &= \expct{\X}{\left(g(\W_1\tr\X)-g_N(\W_1\tr\X,\omega)\right)^2} \\
&= \expct{\Y}{\left(g(\Y)-g_N(\Y,\omega)\right)^2}. \label{eq:expct_GYw}
\end{align}
The last step is an application of Lemma~\ref{lem:x_yz}.
Note that for fixed $\y\in\Yset$ (and variable $\omega\in\Omega$) it holds that
\begin{align}
\expct{}{g_N(\y,\cdot)} &= \int_{\Omega}{\frac{1}{N}{\sum_{j=1}^{N}{f\left(\xyzW{\y}{\Z^\y_j(\omega)}{}\right)}}} \dint{\Prob{}(\omega)} \\
&= \int_{\Omega}{f\left(\xyzW{\y}{\Z^\y(\omega)}{}\right) \dint{\Prob{}(\omega)}} \label{eq:z_iid} \\
&= \int_{\Zset}{f(\xyz) \dint{\probthree{\Z|\Y}{\y}{\z}}} \\
&= g(\y) \label{eq:expct_gN}.
\end{align}
In \eqref{eq:z_iid}, we use that $\Z^\y_j$, $j\in\enum{n},$ are independent and identically distributed.
Taking expectations w.r.t. $\omega$ for the expression in \eqref{eq:expct_GYw} gives
\begin{align}
&\expct{}{\expct{\Y}{\left(g(\Y)-g_N(\Y,\cdot)\right)^2}} = \expct{\Y}{\expct{}{\left(g(\Y)-g_N(\Y,\cdot)\right)^2}} \label{eq:fubini} \\
&\qquad= \expct{\Y}{\var{}{g_N(\Y,\cdot)}} \label{eq:var_gN} \\
&\qquad= \frac{1}{N^2} \sum_{j=1}^{N}{\expct{\Y}{\var{\Z}{f\left(\xyzW{\Y}{\Z^\Y}{}\right)}}} \\
&\qquad= \frac{1}{N} \expct{\Y}{\var{\Z}{f\left(\xyzW{\Y}{\Z^\Y}{}\right)}} \label{eq:z_iid_2} \\
&\qquad= \frac{1}{N} \expct{\Y}{\expct{\Z}{\left(f\left(\xyzW{\Y}{\Z^\Y}{}\right)-g(\Y)\right)^2}} \\
&\qquad= \frac{1}{N} \expct{\X}{(f(\X)-g(\W_1\tr\X))^2} \label{eq:x_yz_2} \\
&\qquad= \frac{1}{N} \expct{\X}{\left(f(\X)-f_g(\X)\right)^2} \\
&\qquad\leq \frac{C_{\ref{thm:expct_fg_f}}}{N} \left(\lambda_{k+1}+\cdots+\lambda_n\right) \label{eq:expct_fg_f}.
\end{align}
Fubini's theorem is applied in \eqref{eq:fubini} since $g_N$ is product measurable due to Lemma~\ref{lem:gN_meas}.
The result in \eqref{eq:expct_gN} justifies \eqref{eq:var_gN}.
In \eqref{eq:z_iid_2}, we reiterated that $\Z^\y_j$, $j\in\enum{n}$, are independent and identically distributed for a fixed $\y\in\Yset$.
Lemma~\ref{lem:x_yz} with $h(\x)\defas (f(\x)-g(\W_1\tr\x))^2$ leads to \eqref{eq:x_yz_2}.
The last equation in \eqref{eq:expct_fg_f} follows from Lemma~\ref{thm:expct_fg_f}.
\end{proof}

The number of samples $N>0$ in the approximating sum shows up in the bound's denominator which is common for Monte Carlo type approximations (the root mean squared error is $\mathcal{O}(N^{-1/2})$).

\subsection*{Stability}
In practice, the matrix $\C$ in \eqref{eq:C} and its eigendecomposition giving the active subspace are only approximately available.
A well-investigated way to get an approximation is through a finite Monte Carlo sum \cite{constantine2014computing,holodnak2018probabilistic,lam2018multifidelity}.
Independent of the concrete type of approximation, only a perturbed representation of the active and inactive subspaces is available, as denoted here by
\begin{equation}
	\hat{\W} = [\Wpert_1\;\;\Wpert_2].
\end{equation}
This subsection is dedicated to the discussion of MSE analysis for perturbations.
We will repeat the behavior of active subspaces with respect to perturbations from \cite{constantine2014active} and extend theorems where necessary.
In the succeeding expressions, we will denote perturbed terms with a hat ( $\hat{}$ ).
For the sake of clarity, let us recall the definitions of the approximating function and its Monte Carlo version for the context of perturbed quantities.
Analogous to the context without perturbations, the domain of $\hat{g}$ is denoted by $\hat{\Yset} \defas \Wpert_1\tr\Xset$.
Let us define
\begin{equation}
	\label{eq:g_pert}
	\hat{g}(\ypert) \defas \int_{\R^{n-k}}{f(\xyzW{\ypert}{\zpert}{\Wpert})\,\pzypert(\zpert|\ypert) \dint{\zpert}}, \qquad f_{\hat{g}}(\x) \defas \hat{g}(\Wpert_1\tr\x)
\end{equation}
and
\begin{equation}
	\label{eq:gNpert}
	\hat{g}_N(\ypert,\cdot) \defas \frac{1}{N} \sum_{j=1}^{N}{f(\xyzW{\ypert}{\Zpert^\ypert_j(\cdot)}{\Wpert})}, \quad \Zpert^\ypert_j \sim \Prob{\Zpert|\Ypert}^\ypert, \qquad f_{\hat{g}_N}(\x,\cdot) \defas \hat{g}_N(\Wpert_1\tr\x,\cdot)
\end{equation}
for $\y\in\hat{\Yset}$ and $\x\in\Xset$.
Note that it is actually enough to integrate over $\hat{\Zset} \defas \Wpert_2\tr\Xset$ in \eqref{eq:g_pert}.
Let $\norm{\cdot}$ denote the Euclidean norm throughout the rest of the manuscript and assume that
\begin{equation}
	\label{eq:pert_bound}
	\norm{\W-\hat{\W}} \leq \epsilon
\end{equation}
for some $\epsilon>0$.

For the subsequent statements, we need a small helping lemma.
It is already stated in \cite[Lemma~3.4]{constantine2014active}; however, our proof is only slightly different (there appears a $\leq$ in \eqref{eq:norm_W1tr_W2} instead of $=$).
\begin{lemma}
\label{lem:W1W2}
Under the assumption in \eqref{eq:pert_bound}, it holds that
\begin{equation}
	\norm{\W_2\tr\Wpert_2} \leq 1 \text{,} \quad \norm{\W_1\tr\Wpert_2} \leq \epsilon, \quad \text{and} \quad \norm{\hat{\W}_2\tr\W_1}\leq\epsilon.
\end{equation}
\end{lemma}
\begin{proof}
By orthogonality of the columns of $\W_2$ and $\Wpert_2$, it holds that
\begin{equation}
	\norm{\W_2\tr\Wpert_2} \leq \norm{\W_2\tr}\norm{\Wpert_2} = 1.
\end{equation}
Additionally,
\begin{align}
	\label{eq:norm_W1tr_W2}
	\norm{\W_1\tr\Wpert_2} &= \norm{\W_1\tr(\Wpert_2-\W_2)} \leq \norm{\W_1\tr}\norm{\Wpert_2-\W_2} = \norm{\Wpert_2-\W_2} \leq \epsilon.
\end{align}
The last line conforms to the desired result.
\end{proof}

For the sake of completeness, a similar result as in Lemma~\ref{thm:expct_fg_f}, but in the context of perturbations, is presented in the succeeding theorem.
\begin{theorem}
\label{thm:f_fghat}
Under the assumptions of Theorem~\ref{thm:expct_fg_f}, it holds that
\begin{equation}
	\expct{\X}{\left(f(\X)-f_{\hat{g}}(\X)\right)^2} \leq C_{\ref{thm:expct_fg_f}} \left(\epsilon(\lambda_1+\cdots+\lambda_k)^{1/2} + (\lambda_{k+1}+\cdots+\lambda_n)^{1/2}\right)^2.
\end{equation}
\end{theorem}
\begin{proof}
The proof is found \cite[Theorem~3.5]{constantine2014active} which uses the chain rule for calculating $\grad_{\hat{\z}}f = \W_1\tr\hat{\W}_2\grad_\y f + \W_2\tr\hat{\W}_2\grad_\z f$ and the result for Lemma~\ref{lem:W1W2}.
\end{proof}

In this bound, the (large) eigenvalues of the active subspace play some role as well.
Fortunately, they show up with the factor $\epsilon>0$ being itself a bound on the Euclidean norm of $\W-\hat{\W}$.
If this deviation of $\W$ and $\hat{\W}$ is sufficiently small, then the impact of the larger eigenvalues becomes rather small, and the bound is dominated by the eigenvalues corresponding to the inactive subspace.

As a consequence of our framework, a perturbed version of Theorem~\ref{thm:var_mc} can also be proved.
We will stay with the previous notations, $\mathbb{E}_\Y$ and $\mathbb{E}_\Z$, to denote expectations involving the perturbed random variables $\hat{\Y}$ and $\hat{\Z}$, respectively.

\begin{theorem}
\label{thm:var_mc_pert}
Under the assumptions of Theorem~\ref{thm:expct_fg_f}, it holds that
\begin{align}
	\expct{}{\expct{\X}{\left(f_{\hat{g}}(\X)-f_{\hat{g}_N}(\X,\cdot)\right)^2}} &= \expct{}{\expct{\Y}{\left(\hat{g}(\hat{\Y})-\hat{g}_N(\hat{\Y},\cdot)\right)^2}} \\
	&\leq \frac{C_{\ref{thm:expct_fg_f}}}{N} \left(\epsilon(\lambda_1+\cdots+\lambda_k)^{1/2} + (\lambda_{k+1}+\cdots+\lambda_n)^{1/2}\right)^2.
\end{align}
\end{theorem}
\begin{proof}
The proof follows the same arguments as in Theorem~\ref{thm:var_mc} with the exception of the last step, which uses Theorem~\ref{thm:f_fghat} instead.
\end{proof}

Eventually, according to Theorem~3.6 in \cite{constantine2014active} (see also \cite{constantine2014erratum}), we can give an upper bound on the expectation of the mean squared error between $f$ and $f_{\hat{g}_N}$.
\begin{theorem}
\label{thm:mse_f_fgN}
Under the assumptions of Theorem~\ref{thm:expct_fg_f}, it holds that
\begin{equation}
	\expct{}{\expct{\X}{\left(f(\X)-f_{\hat{g}_N}(\X,\cdot)\right)^2}} \leq C_{\ref{thm:expct_fg_f}} \left(1+N^{-1/2}\right)^2 \left(\epsilon(\lambda_1+\cdots+\lambda_k)^{1/2} + (\lambda_{k+1}+\cdots+\lambda_n)^{1/2}\right)^2.
\end{equation}
\end{theorem}
The bound shows that the number of samples $N>0$ for the Monte Carlo approximation has little influence on the approximation quality of $f_{\hat{g}_N}$.
The eigenvalues corresponding to the inactive subspace are however, more crucial.

\subsection*{Numerical experiment}
To verify the previous statements numerically, we are going to lay down a simple example that is computationally cheap to analyze.
Let us consider a quadratic function of interest
\begin{equation}
	f : \Xset\to\R, \quad \x\mapsto\frac{1}{2}\x\tr\A\x
\end{equation}
in $n=10$ variables for a \textit{symmetric} matrix $\A\in\R^{n\times n}$.
In addition, we will assume a standard normal distribution on the domain of $f$, \ie $\Prob{\X}\defas\mathcal{N(\vec{0},\I)}$, and thus $\Xset=\R^n$.
To calculate the active subspace of this function, we need the gradient of $f$, which is
\begin{equation}
	\grad f(\x) = \A\x, \quad \x\in\Xset.
\end{equation}
Now, we can compute
\begin{align}
	\C &= \int_{\Xset}{\grad f(\x)\grad f(\x)\tr \px(\x)\dx} \\
	&= \A\left(\int_{\Xset}{\x\x\tr\px(\x)\dx}\right)\A\tr \\
	&= \A^2.
\end{align}
In order to get a good test example, let us choose
\begin{equation}
	\A \defas \W\MeigVals^{1/2}\W\tr,
\end{equation}
where $\W\in\R^{n\times n}$ is an arbitrary orthogonal matrix and $\MeigVals\in\R^{n\times n}$ a diagonal matrix containing descending eigenvalues having a spectral gap of almost two orders of magnitude after the second eigenvalue, \ie
\begin{align}
	\MeigVals &\defas \diag{10^4,10^{3.8},10^{2},10^{1.75},\dots,10^{0.25}} \\
	&= \begin{pmatrix}\MeigVals_1 & \\ & \MeigVals_2\end{pmatrix}.
\end{align}
The diagonal submatrices $\MeigVals_1\in\R^{k\times k}$, $k=2$, and $\MeigVals_2\in\R^{(n-k)\times(n-k)}$ contain eigenvalues from the active and inactive subspaces, respectively.
\begin{figure}
	\centering
	\includegraphics[scale=0.5]{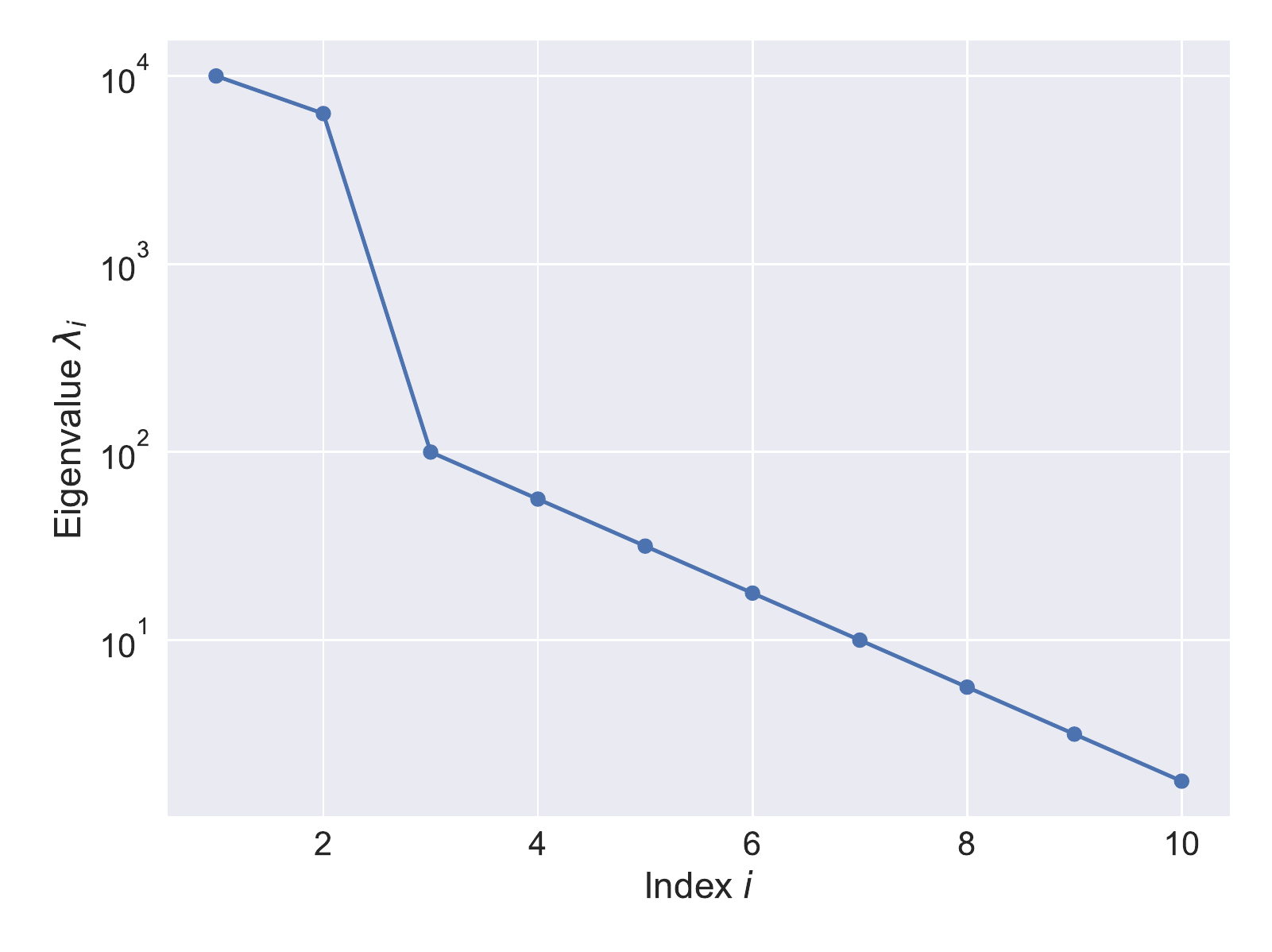}
	\caption{Decay of eigenvalues occurring in the numerical experiment.}
	\label{fig:eigvals_mse}
\end{figure}
The eigenvalues are plotted in Figure~\ref{fig:eigvals_mse}.
The function $f$, in terms of the active variable $\y\in\R^k$ and the inactive variable $\z\in\R^{n-k}$, can be computed explicitly as
\begin{align}
	f(\xyz) &= \frac{1}{2}\xyz\tr \A \,\xyz \\
	&= \frac{1}{2}(\y\tr\W_1\tr\W\MeigVals^{1/2}\W\tr\W_1\y + \z\tr\W_2\tr\W\MeigVals^{1/2}\W\tr\W_1\y \\
	& \qquad + \y\tr\W_1\tr\W\MeigVals^{1/2}\W\tr\W_2\z + \z\tr\W_2\tr\W\MeigVals^{1/2}\W\tr\W_2\z) \\
	&= \frac{1}{2}\left(\y\tr\begin{bmatrix}\I & \vec{0}\end{bmatrix}\MeigVals^{1/2}\begin{bmatrix}\I\\\vec{0}\end{bmatrix}\y + \z\tr\begin{bmatrix}\vec{0} & \I\end{bmatrix}\MeigVals^{1/2}\begin{bmatrix}\I\\\vec{0}\end{bmatrix}\y\right. \\
	& \qquad + \left.\y\tr\begin{bmatrix}\I & \vec{0}\end{bmatrix}\MeigVals^{1/2}\begin{bmatrix}\vec{0}\\\I\end{bmatrix}\z + \z\tr\begin{bmatrix}\vec{0} & \I\end{bmatrix}\MeigVals^{1/2}\begin{bmatrix}\vec{0}\\\I\end{bmatrix}\z\right) \\
	&= \frac{1}{2}\left(\y\tr\MeigVals_1^{1/2}\y + \z\tr\MeigVals_2^{1/2}\z\right).
\end{align}
It follows that $g$, defined in \eqref{eq:gy}, can be written as
\begin{align}
	g(\y) &= \frac{1}{2}\y\tr\MeigVals_1^{1/2}\y +  \frac{1}{2}\int_{\Zset}{\z\tr\MeigVals_2^{1/2}\z\;\pzy(\z|\y)\dz} \\
	&= \frac{1}{2}\y\tr\MeigVals_1^{1/2}\y + \frac{1}{2}\sum_{i=1}^{n-k}{\eigVal_{k+i}^{1/2}\int_{\Zset}{\z_i^2\;\pzy(\z|\y)\dz}} \\
	&= \frac{1}{2}\y\tr\MeigVals_1^{1/2}\y + \frac{1}{2}\trace{\MeigVals_2^{1/2}}
\end{align}
for $\y\in\Yset$.
The last line uses the fact that $\pzy$ is again a standard normal density since $\px$ is rotationally symmetric and $\x \mapsto (\y, \, \z)$ is an orthogonal mapping.
Note that $g(\y)$ depends only on eigenvalues corresponding to the active subspace if eigenvalues in $\MeigVals_2$ are all zero.
Similarly, the Monte Carlo approximation of $g$, $g_N$, defined in \eqref{eq:gNy}, is
\begin{equation}
	g_N(\y) = \frac{1}{2}\y\tr\MeigVals_1^{1/2}\y + \frac{1}{2N}\sum_{j=1}^{N}{\left(\Z^\y_j\right)\tr \MeigVals_2^{1/2} \Z^\y_j}
\end{equation}
for $\y\in\Yset$, where $\Z^\y_j\sim\Prob{\Z|\Y}^\y = \mathcal{N}(\vec{0},\I)$.
First, we would want to examine the convergence behavior of the mean squared error $\MSE_{f_g,f_{g_N}}\defas\expct{\X}{(f_g(\X)-f_{g_N}(\X,\cdot))^2}$ between $f_g$ and $f_{g_N}$ in the number of samples $N$ used for the approximation of $f_{g_N}$.
For fixed $N>0$, we will approximate the mean squared error by
\begin{equation}
	\MSE_{f_g,f_{g_N}} \approx \frac{1}{N_\x}\sum_{i=1}^{N_\x}{(f_g(\X_i)-f_{g_N}(\X_i,\cdot))^2},
\end{equation}
where $\X_i\sim\Prob{\X}$, $i\in\enum{N_\x}$, are $N_\x>0$ random values in the domain of $f$.
We can choose $N_\x=10^4$ to get a sufficiently accurate approximation.
Since the mean squared error is random, we can compute $N_\z=10^3$ realizations of it to approximate
\begin{equation}
	\expct{}{\MSE_{f_g,f_{g_N}}}
\end{equation}
which is the quantity we found a bound for in Theorem~\ref{thm:var_mc}.
Additionally, we can investigate the \textit{coefficient of variation}, denoted by $\mathbb{C}\textnormal{V}$, of $\MSE_{f_g,f_{g_N}}$ defined by
\begin{equation}
	\cv{}{\MSE_{f_g,f_{g_N}}} \defas \frac{\std{}{\MSE_{f_g,f_{g_N}}}}{\expct{}{\MSE_{f_g,f_{g_N}}}},
\end{equation}
where $\mathbb{S}\textnormal{td}$ denotes the standard deviation.
We will run the experiment for $N=2,5,10,20,50,100$ samples.
Same steps are follows when investigating the random variable $\MSE_{f,f_{g_N}}\defas\expct{\X}{(f(\X)-f_{g_N}(\X,\cdot))^2}$.
Theorem~\ref{thm:mse_f_fgN} provides an upper bound on its expectation value.
The computational results are plotted in Figure~\ref{fig:mse}.
These results verify the first order behavior in $N$ of $\mathbb{E}[\MSE_{f_g,f_{g_N}}]$ and show furthermore that the variation of the random variables $\MSE_{f,f_{g_N}}$ and $\MSE_{f_g,f_{g_N}}$ is nearly constant w.r.t. $N$.
This information is valuable since it means that the consequences of regarding $\MSE_{f,f_{g_N}}$ and $\MSE_{f_g,f_{g_N}}$ as deterministic are limited.
In addition, the left plot confirms the fact that an increasing number of samples has a decreasing effect on the (expectation of the) mean squared error between $f$ and $f_{g_N}$.
\begin{figure}
	\begin{subfigure}{0.49\textwidth}
		\centering
		\includegraphics[width=\textwidth]{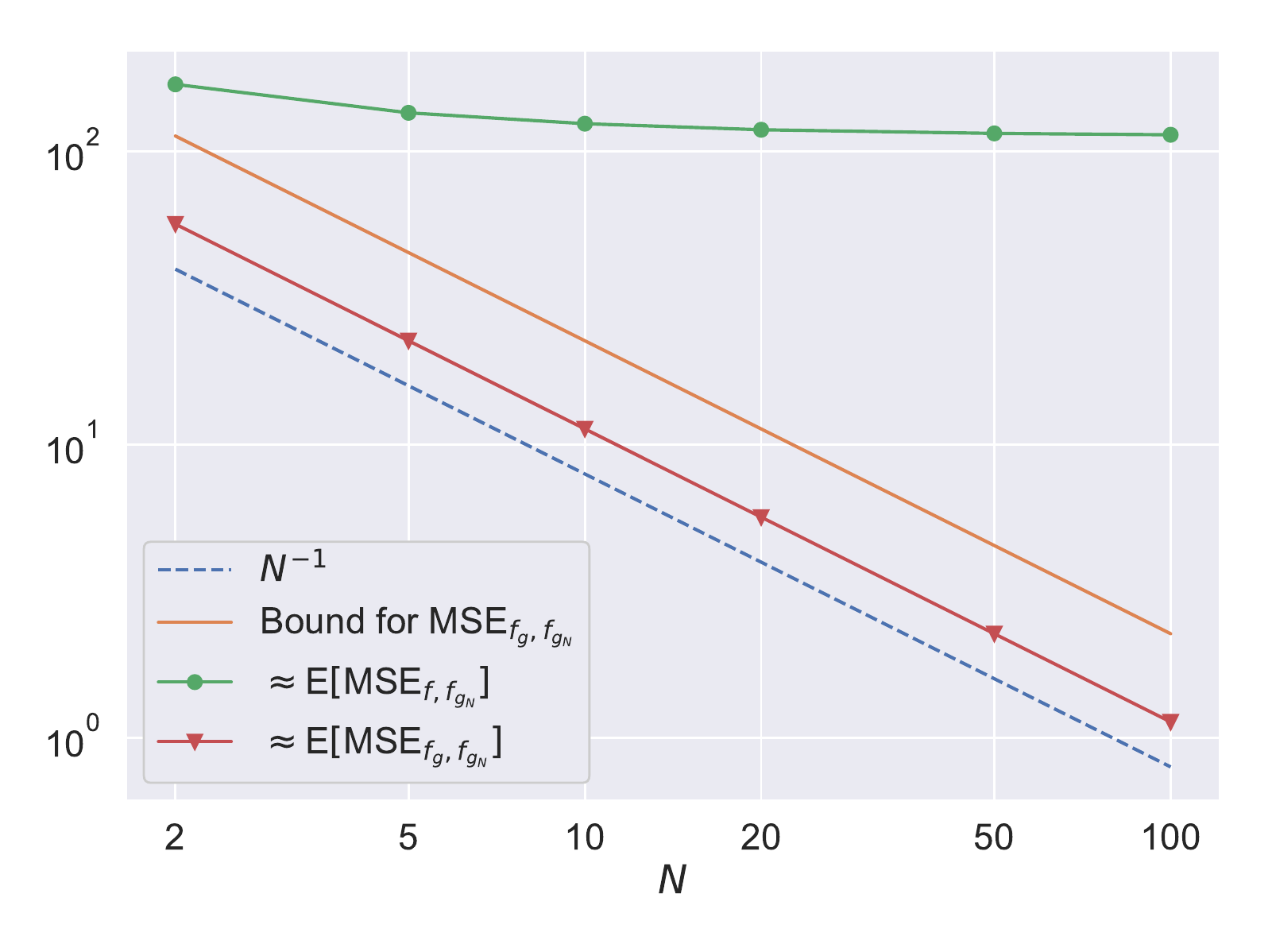}
	\end{subfigure}
	\begin{subfigure}{0.49\textwidth}
		\centering
		\includegraphics[width=\textwidth]{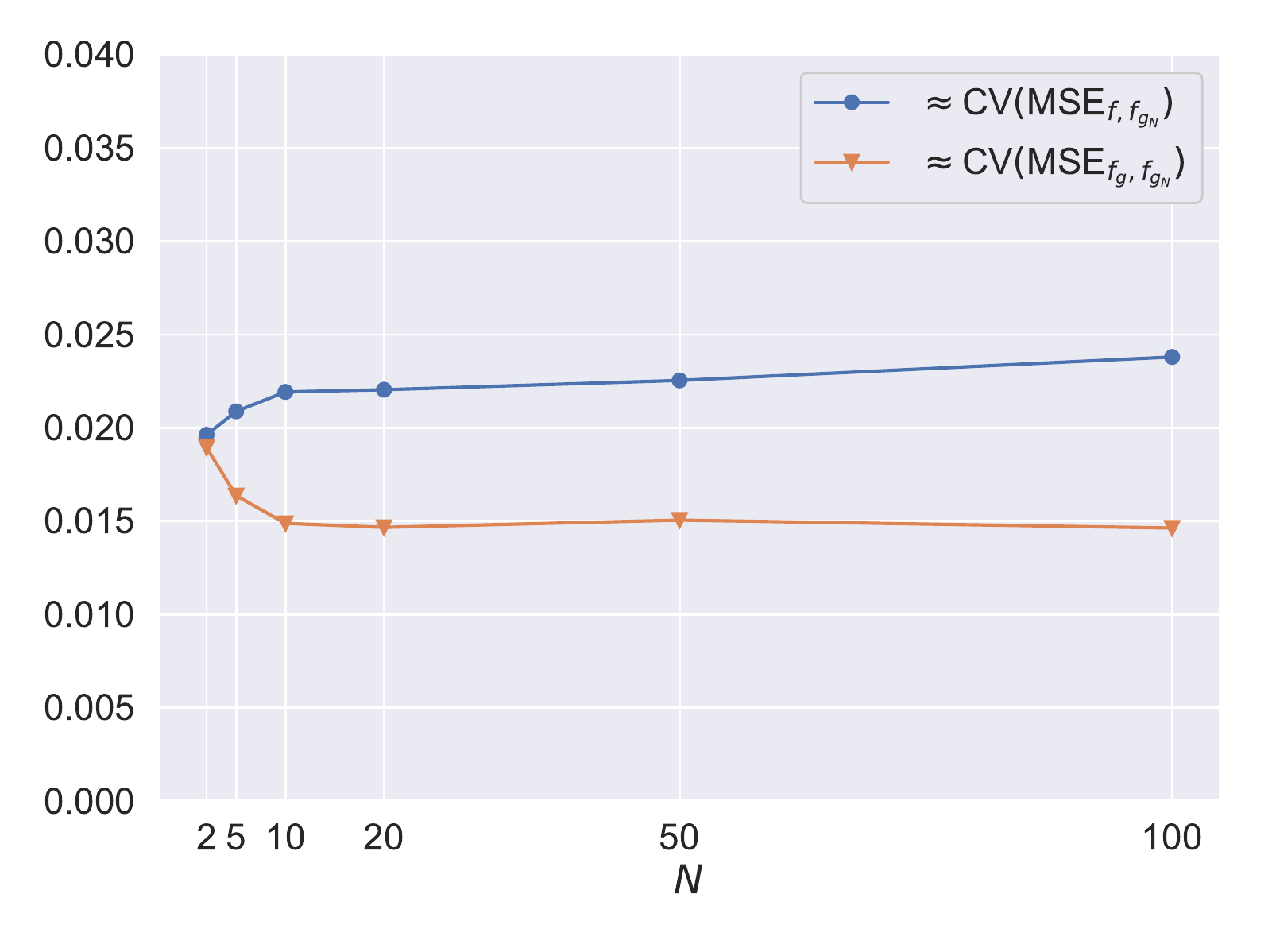}
	\end{subfigure}
%
	\caption{The left plot confirms the linear decay in $N$ of the expectation of the mean squared error $\mathbb{E}[\MSE_{f_g,f_{g_N}}]$.
			It also shows that increasing the number of samples $N$ does not have much effect since the eigenvalues from the inactive subspace are dominating.
			The randomness of $\MSE_{f_g,f_{g_N}}$ is illustrated in the right figure by showing that $\mathbb{C}\textnormal{V}(\MSE_{f_g,f_{g_N}})$ and $\mathbb{C}\textnormal{V}(\MSE_{f,f_{g_N}})$ are not zero.
			}
	\label{fig:mse}
\end{figure}

\section{Bayesian inversion in the active subspace}
\label{sec:bayinv_as}
In \cite{constantine2016accelerating}, it was possible to accelerate the mixing in a Metropolis-Hastings algorithm that produces samples of a posterior distribution from a Bayesian setting.
Let us first define the setup in the context of Bayesian inversion before we come to the more interesting and critical point relevant to the current framework.

In Bayesian inversion, one tries to infer parameters $\x\in\R^n$ of a model $\model:\Xset\to\R^{n_\data}$ in a statistical setting \cite{kaipio2006statistical}.
The theory is not restricted to the parameter and data space we use here for simplicity, but can be extended to a much more general setting \cite{stuart2010inverse}.
For example, the outcome of $\model$ can be the solution of a PDE applied to a linear functional called the \textit{Quantity of Interest} (QoI).
The parameters are regarded as random variables and are thus able to model uncertainty.
We begin by assuming a \textit{prior probability distribution}, induced by a density function $\prior$, on the space of the parameters.
The prior is updated to the \textit{posterior distribution} $\posterior$ by incorporating \textit{data} $\data\in\R^{n_\data}$ which is also treated as a random variable.
We then model the data by $\data=\model(\X)+\bm{\eta}$, where $\bm{\eta}\sim\mathcal{N}(0,\cov)$ is additive Gaussian noise modeling measurement error with a covariance matrix $\cov\in\R^{n_\data\times n_\data}$.
The update is formulated by \textit{Bayes' Theorem} which makes a statement about the conditional probability of $\X$ given $\data$.
That is,
\begin{equation}
	\label{eq:def_post}
	\posterior(\x|\data) \defas \frac{\like(\data|\x)\prior(\x)}{\int_{\R^n}{\like(\data|\x')\prior(\x')\dx'}}
\end{equation}
for $\x\in\Xset$.
The concrete expression of the likelihood $\like$ is determined by the model for the measurement error.
In our case, i.\,e. assuming additive Gaussian noise for the measurement error, it holds that for $\x\in\Xset$
\begin{equation}
	\like(\x|\data) \propto \exp\left(-\frac{1}{2}\norm{\data-\model(\x)}^2_\cov\right) \asdef \exp\left(-f_\data(\x)\right).
\end{equation}
The function $f_\data(\x)\defas\frac{1}{2}\norm{\data-\model(\x)}^2_\cov \defas \frac{1}{2}\norm{\cov^{-1/2}(\data-\model(\x))}^2$, $\x\in\Xset$, is called the \textit{data misfit function}.

\textit{Markov chain Monte Carlo} (MCMC) \cite{brooks2011handbook} methods are a well-known technique for interrogating the posterior distribution.
MCMC constructs a Markov chain such that its stationary distribution is the one we want to sample from, i.\,e. the posterior in this case.
One popular MCMC algorithm is the \textit{Metropolis-Hastings algorithm} \cite{brooks2011handbook} which is also used in \cite{constantine2016accelerating}.

Metropolis-Hastings can be computationally inefficient in high-dimensional parameter spaces.
One opportunity to increase efficiency that is presented in \cite{constantine2016accelerating} is dimension reduction by active subspaces.
That is, our function of interest $f$ from the active subspace context is chosen to be the data misfit function $f_\data$ from the Bayesian setting, \ie $f(\x)\defas f_\data(\x)$, $\x\in\Xset$.
Intuitively, the active subspace of $f_\data$ contains directions in the parameter space that are informed by data $\data$ very well.
The prior plays the role of the given density function used for weighting the gradients in \eqref{eq:C}, i.\,e. $\px \defas \prior$.
Hence, the posterior on the whole space is given by
\begin{equation}
	\posterior(\x) = \frac{\exp(-f(\x))\px(\x)}{Z}
\end{equation}
for $\x\in\Xset$, where $Z\defas\int_{\R^n}{\exp(-f(\x'))\px(\x')\dx'}$ is a normalizing constant necessary to get a proper probability density function with unit mass.
We can remove the conditioning on $\data$ explicitly to keep the notation clear.
Respective versions for approximate posteriors using approximations $f_g$ and $f_{g_N}$ are defined through
\begin{equation}
	\rho_{\textnormal{post},g}(\x) \defas \frac{\exp(-f_{g}(\x))\px(\x)}{Z_{g}} \quad\text{and}\quad \rho_{\textnormal{post},g_N}(\x,\cdot) \defas \frac{\exp(-f_{g_N}(\x,\cdot))\px(\x)}{Z_{g_N}(\cdot)}.
\end{equation}
for $\x\in\Xset$.

Consequently, we also want to regard results involving perturbed versions of the posterior which are defined by
\begin{equation}
\label{eq:post_gN_pert}
\rho_{\textnormal{post},\hat{g}}(\x) \defas \frac{\exp(-f_{\hat{g}}(\x))\px(\x)}{Z_{\hat{g}}} \quad\text{and}\quad \rho_{\textnormal{post},\hat{g}_N}(\x,\cdot) \defas \frac{\exp(-f_{\hat{g}_N}(\x,\cdot))\px(\x)}{Z_{\hat{g}_N}(\cdot)}
\end{equation}
for $\x\in\Xset$.
Note that $\rho_{\textnormal{post},g_N}(\x,\cdot)$ and $\rho_{\textnormal{post},\hat{g}_N}(\x,\cdot)$ are random variables for each $\x\in\Xset$, as well as the normalizing constants $Z_{g_N}(\cdot)$ and $Z_{\hat{g}_N}(\cdot)$.

The result that we want to restate here gives an upper bound on the (expected) Hellinger distance between the true posterior and its approximation via $\hat{g}_N$.
Let us investigate a bound involving the approximation with $\hat{g}$, \ie the perturbed version of $g$ but without randomness through the next theorem, which is taken from \cite[Theorem~3.1]{constantine2016accelerating}.
Its proof is attached in Appendix~\ref{app:proof_hell_post_gpert} for the sake of completeness and uses results from Section~\ref{sec:approx}.
\begin{theorem}
\label{thm:hell_post_gpert}
Under the assumptions of Theorem~\ref{thm:expct_fg_f}, it holds that
\begin{equation}
d_{\textnormal{H}}(\rho_{\textnormal{post}},\rho_{\textnormal{post},\hat{g}}) \leq \sqrt{C_{\ref{thm:expct_fg_f}}} L \left(\epsilon(\lambda_1+\cdots+\lambda_k)^{1/2} + (\lambda_{k+1}+\cdots+\lambda_n)^{1/2}\right),
\end{equation}
where 
\begin{equation}
L^2 \defas \frac{1}{8}\left(Z \exp\left(-\int_{\Xset}{f(\x)\px(\x)\dx}\right)\right)^{-1/2} > 0.
\end{equation}
\end{theorem}

Note that, contrary to \cite[Theorem~3.1]{constantine2016accelerating}, the Poincaré constant appears as a square root $\sqrt{C_{\ref{thm:expct_fg_f}}}$ (instead of $C_{\ref{thm:expct_fg_f}}$).
This is a similar result as in Theorem~\ref{thm:f_fghat} asserting that the Hellinger distance between the true posterior and the one approximated with $\hat{g}$ is dominated by the eigenvalues from the inactive subspace if the Euclidean norm of $\W-\hat{\W}$ is small.
That is, the smaller these eigenvalues are, the better is the approximation of the posterior with perturbed $\hat{g}$, on average.

For the approximation involving a random $\hat{g}_N$, note that the Hellinger distance $d_{\textnormal{H}}(\rho_{\textnormal{post},g},\rho_{\textnormal{post},g_N}(\cdot))$ is also a random variable, \ie we can, for example, make statements on its expectation $\mathbb{E}_\Z$.
A corresponding statement is given in the next theorem.
\begin{theorem}
\label{thm:hell_post_gpert_gpertN}
Under the assumptions of Theorem~\ref{thm:expct_fg_f}, It holds that
\begin{equation}
\expct{}{d_{\textnormal{H}}(\rho_{\textnormal{post},\hat{g}},\rho_{\textnormal{post},\hat{g}_N})} \leq \sqrt{\frac{C_{\ref{thm:expct_fg_f}}}{N}} L \left(\epsilon(\lambda_1+\cdots+\lambda_k)^{1/2}+(\lambda_{k+1}+\cdots+\lambda_n)^{1/2}\right), \label{eq:bound_pert_hell}
\end{equation}
where $L>0$ is the same constant as in Theorem~\ref{thm:hell_post_gpert}.
\end{theorem}
\begin{proof}
The proof is similar to the one in Theorem~\ref{thm:hell_post_gpert}.
The main difference is the usage of the Cauchy--Schwarz inequality and Theorem~\ref{thm:var_mc_pert} in the last step.
Specifically, the last step is
\begin{align}
\expct{}{\dhell{\posteriorgpert}{\posteriorgNpert}}^2 &\leq \frac{1}{8} Z_{\hat{g}}^{-1/2} \, \expct{}{\left(Z_{\hat{g}_N}^{-1/2} \, \expct{\X}{(f_{\hat{g}}(\X)-f_{\hat{g}_N}(\X,\cdot))^2}\right)^{1/2}}^2 \\
&\leq \frac{1}{8} Z_{\hat{g}}^{-1/2} \, \expct{}{Z_{\hat{g}_N}^{-1/2}} \expct{}{\expct{\X}{(f_{\hat{g}}(\X)-f_{\hat{g}_N}(\X,\cdot))^2}} \\
&\leq \frac{C_{\ref{thm:expct_fg_f}}}{N} L_{\hat{g},\hat{g}_N}^2 \left(\epsilon(\lambda_1+\cdots+\lambda_k)^{1/2}+(\lambda_{k+1}+\cdots+\lambda_n)^{1/2}\right)^2,
\end{align}
where 
\begin{equation}
L_{\hat{g},\hat{g}_N}^2 \defas \frac{1}{8} Z_{\hat{g}}^{-1/2} \expct{}{Z_{\hat{g}_N}^{-1/2}} > 0.
\end{equation}
The result follows by noting that
\begin{align}
	\expct{}{Z_{\hat{g}_N}} &\geq \exp\left(-\int_{\hat{\Yset}}{\expct{}{\hat{g}_N(\ypert,\cdot)} \pypert(\ypert)\dint{\ypert}}\right) \label{eq:lb_ZgN} \\
	&= \exp\left(-\int_{\hat{\Yset}}{\hat{g}(\ypert) \, \pypert(\ypert)\dint{\ypert}}\right) \\
	&= \exp\left(-\int_{\Xset}{f(\x) \, \px(\x)\dx}\right).
\end{align}
In \eqref{eq:lb_ZgN}, we changed integrals based on the result of Lemma~\ref{lem:gN_meas} (for perturbed quantities).
\end{proof}

According to Theorem~3.1 in \cite{constantine2016accelerating}, we can find an upper bound on the expectation of the Hellinger distance between the true posterior and $\rho_{\textnormal{post},\hat{g}_N}$ using the triangle equality.
\begin{theorem}
\label{thm:hell_truepost_mc_pert}
Under the assumptions of Theorem~\ref{thm:expct_fg_f}, it holds that
\begin{equation}
	\expct{}{d_{\textnormal{H}}(\rho_{\textnormal{post}},\rho_{\textnormal{post},\hat{g}_N})} \leq \sqrt{C_{\ref{thm:expct_fg_f}}} L \left(1+N^{-1/2}\right)\left(\epsilon(\lambda_1+\cdots+\lambda_k)^{1/2}+(\lambda_{k+1}+\cdots+\lambda_n)^{1/2}\right),
\end{equation}
where $L>0$ is the same constant as in Theorem~\ref{thm:hell_post_gpert}.
\end{theorem}

Similar to Theorem~\ref{thm:mse_f_fgN}, increasing the number of samples $N$ to gain accuracy will not have a large effect if the eigenvalues of the inactive subspace are too large and hence dominating.

\section{Summary}
\label{sec:summary}
This manuscript proposed a comprehensive probabilistic setting for approximating functions in active subspaces.
This was necessary to show that a certain expression for the mean squared error of a conditional expectation and its Monte Carlo approximation is a random term; thus, suggesting extensions of the analyses in \cite{constantine2014active,constantine2016accelerating} to a truly probabilistic setting.

Section~\ref{sec:prob_form} formulated the problem in general and motivates the reason for subsequent discussions.
Section~\ref{sec:as} introduced the notion of an active subspace and proved fundamental lemmas required for rigorous reasoning on latter details.
Section~\ref{sec:approx} defined the conditional expectation of a function of interest $f$ over the inactive subspace and used it as its approximation.
The randomness of the mean squared error between the conditional expectation and its Monte Carlo approximation brought us to extend results from \cite{constantine2014active}.
The results were verified numerically through a simple test example.
Figures also supported the presence of randomness and displayed the statistical properties, \eg expectations and coefficients of variation, of random terms.
Lastly, Section~\ref{sec:bayinv_as} discussed the applications of theorems from Section~\ref{sec:approx} to restate results from \cite{constantine2016accelerating} in the context of Bayesian inversion.
Within this context, the Hellinger distance of an exact Bayesian posterior distribution and its approximation using active subspaces is bounded from above by eigenvalues from the inactive subspace.
Since this expression, using a Monte Carlo approximation, is also random, the previous results were utilized to confirm a similar bound from \cite{constantine2016accelerating}.
\section*{Acknowledgments}
The author would like to acknowledge the assistance of the following people in the TUM community particularly: the significant help and constructive advice of Prof.~Barbara Wohlmuth, the support of the Chair for Numerical Mathematics, and the contributions of David Criens and Dominik Schmid concerning questions on measurable functions and measurable sets.

Furthermore, an important remark of Olivier Zahm (INRIA) on Poincaré and logarithmic Sobolev inequalities is gratefully acknowledged.

\appendix
\section{Proof of Theorem~\ref{thm:hell_post_gpert}}
\label{app:proof_hell_post_gpert}
\begin{proof}
Repeating the steps from \cite[Theorem 3.1]{constantine2016accelerating} gives
\begin{align}
d_\textnormal{H}&\left(\rho_{\textnormal{post}},\rho_{\textnormal{post},\hat{g}}\right)^2 = \frac{1}{2} \int_{\Xset}{\left[\rho_{\text{post}}(\x)^{1/2} - \rho_{\text{post},\hat{g}}(\x)^{1/2}\right]^2 \dx} \\
&= \frac{1}{2} \int_{\Xset}{\left[\left(\frac{\exp(-f(\x))\px(\x)}{Z_g}\right)^{1/2} - \left(\frac{\exp(-f_{\hat{g}}(\x))\px(\x)}{Z_{\hat{g}}}\right)^{1/2}\right]^2 \dx} \\
&= \frac{1}{2} \int_{\Xset}{\left[\left(\frac{\exp(-f(\x))}{Z}\right)^{1/2} - \left(\frac{\exp(-f_{\hat{g}}(\x))}{Z_{\hat{g}}}\right)^{1/2}\right]^2 \px(\x) \dx} \\
&= \frac{1}{2\left(ZZ_{\hat{g}}\right)^{1/2}} \left[\int_{\Xset}{\left(\exp(-f(\x))^{1/2} - \exp(-f_{\hat{g}}(\x))^{1/2}\right)^2\px(\x)\dx}\right. \\
&\hspace{4cm}\left. - \left(Z^{1/2} - Z_{\hat{g}}^{1/2}\right)^2 \right] \nonumber \\
&\leq \frac{1}{2\left(ZZ_{\hat{g}}\right)^{1/2}} \int_{\Xset}{\left[\exp(-f(\x))^{1/2} - \exp(-f_{\hat{g}}(\x))^{1/2}\right]^2\px(\x)\dx} \\
&= \frac{1}{2\left(ZZ_{\hat{g}}\right)^{1/2}} \int_{\Xset}{\left[\exp\left(-\frac{f(\x)}{2}\right) - \exp\left(-\frac{f_{\hat{g}}(\x)}{2}\right)\right]^2\px(\x)\dx} \\
&\leq \frac{1}{2\left(ZZ_{\hat{g}}\right)^{1/2}} \int_{\Xset}{\left(\frac{1}{2}\left(f(\x) - f_{\hat{g}}(\x)\right)\right)^2 \px(\x) \dx} \\
&\leq \frac{1}{8\left(ZZ_{\hat{g}}\right)^{1/2}} \int_{\Xset}{\left(f(\x)-f_{\hat{g}}(\x)\right)^2\px(\x) \dx} \\
&= \frac{1}{8\left(ZZ_{\hat{g}}\right)^{1/2}} \expct{\X}{\left(f(\X)-f_{\hat{g}}(\X)\right)^2} \\
&\leq C_{\ref{thm:expct_fg_f}} L_{f,\hat{g}}^2 \left(\epsilon(\lambda_1+\cdots+\lambda_k)^{1/2} + (\lambda_{k+1}+\cdots+\lambda_n)^{1/2}\right)^2, \label{eq:expct_fg_fgN}
\end{align}
where
\begin{equation}
L_{f,\hat{g}}^2 \defas \frac{1}{8}\left(ZZ_{\hat{g}}\right)^{-1/2}.
\end{equation}
The last equation in \eqref{eq:expct_fg_fgN} uses the result of Theorem~\ref{thm:f_fghat}.
The result follows by noting that
\begin{align}
	Z_{\hat{g}} &\geq \exp\left(-\int_{\Xset}{f_{\hat{g}}(\x) \, \px(\x)\dx}\right) \\
	&= \exp\left(-\int_{\hat{\Yset}}{\hat{g}(\ypert) \, \pypert(\ypert)\dint{\ypert}}\right) \\
	&= \exp\left(-\int_{\hat{\Yset}}\left(\int_{\hat{\Zset}}{f(\xyzW{\ypert}{\zpert}{\Wpert}) \, \pzypert(\zpert|\ypert)\dint{\zpert}}\right) \pypert(\ypert) \dint{\ypert}\right) \\
	&= \exp\left(-\int_{\Xset}{f(\x) \, \px(\x)\dx}\right).
\end{align}
\end{proof}
%

\end{document}